\documentclass[letter]{amsart}
\usepackage{amssymb,amsmath}
\usepackage{stmaryrd,paralist}
\usepackage{bbm}
\usepackage{graphics,graphicx}
\usepackage{hyperref}

\usepackage[latin1]{inputenc}
\usepackage{subfigure}
\usepackage[margin=1.4 in]{geometry}
\usepackage{hyperref,color}
\usepackage[dvipsnames]{xcolor}
\graphicspath{{Figures/}} 
\graphicspath{{./Figures/}}

\usepackage{url}

\newcommand{\addllncs}[1]{}

\newtheorem{theo}{Theorem}
\numberwithin{theo}{section}

\newtheorem{lemma}[theo]{Lemma}
\newtheorem{prop}[theo]{Proposition}

\theoremstyle{remark}
\newtheorem{remark}[theo]{Remark}

{\medbreak}

\newcommand{\cacher}[1]{}

\def\cA{\mathcal{A}}
\def\cB{\mathcal{B}}
\def\cC{\mathcal{C}}

\def\cP{\mathcal{P}}
\def\cQ{\mathcal{Q}}
\def\cF{\mathcal{F}}
\def\cG{\mathcal{G}}
\def\cO{\mathcal{O}}

\def\cT{\mathcal{T}}

\def\cQfc{\cQ_{5c}}
\def\cTfc{\cT_{5c}}

\def\wcQfc{\widetilde{\cQ}_{5c}}
\def\wcTfc{\widetilde{\cT}_{5c}}

\def\cFfc{\cF_{5c}}
\def\Ffc{F_{5c}}
\def\Ffco{F_{5co}}
\def\cFco{\cF_{5co}}
\def\cGfc{\cG_{5c}}

\def\btau{\bar{\tau}}

\newcommand{\fig}[3]{\begin{figure}[h!]\begin{center}\includegraphics[#1]{#2.pdf}\end{center}\caption{#3}\label{fig:#2}\end{figure}}

\def\set5{[1:5]}

\def\tx{\tilde{x}}
\def\ty{\tilde{y}}

\def\sx{x^*}
\def\sy{y^*}

\def\tM{\widetilde{M}}
\def\sM{M^*}

\begin{document}

\author[\'Eric Fusy]{\'{E}ric Fusy$^{*}$}
\thanks{$^{*}$LIGM, CNRS, Univ Gustave Eiffel, ESIEE Paris, F-77454 Marne-la-Vall\'ee, France, eric.fusy@univ-eiffel.fr}

\title[A model of trees for 5-connected planar triangulations]{A model of trees for 5-connected planar  triangulations}
\date{\today}

\begin{abstract}
Triangulations of the 5-gon with no separating triangle nor quadrangle, so called 5c-triangulations, are a planar map family closely related to 5-connected planar triangulations. We show that 5c-triangulations are in bijection with 5-regular plane trees satisfying a simple local constraint at inner edges. It yields explicit expressions for the generating functions of rooted 5c-triangulations, and of rooted 5-connected planar triangulations with root-vertex degree~5, these belonging to the same algebraic extension as the generating function of rooted 5-connected planar triangulations computed by Gao, Wanless and Wormald. The bijection also makes it possible to obtain efficient uniform random generation and succinct encoding procedures for 5-connected planar triangulations.
\end{abstract}

\maketitle

\section{Introduction}\label{sec:intro}

A rich literature is devoted to the enumeration of (rooted) planar maps, initiated by the pioneering work of Tutte~\cite{tutte1962census,tutte1963census}. The original method proceeded by showing that the generating function ---with an additional catalytic variable--- satisfies a functional equation, obtained typically by translating a decomposition by root deletion, and then solving the functional equation either by guessing-checking, or by the quadratic method~\cite{Tutte73} and its extensions~\cite{Bousquet-MelouJ06,eynard2016counting}. Another fruitful approach is to relate families of maps via substitution schemes~\cite{tutte1963census,mullin1968enumeration,banderier2001random,gao2001,BouttierG14}, allowing to extract expressions for the generating functions of map families with higher connectivity/girth constraints (whereas catalytic equations for such families are often more difficult to obtain).   
Various families of planar maps can be enumerated 
by these methods, 
 yielding surprisingly simple counting formulas.  
Other map enumeration approaches have been developed, in particular matrix integral techniques~\cite{di19952d}, and closer to our focus, bijective constructions that provide direct combinatorial proofs of the counting formulas, starting with the Cori-Vauquelin bijection~\cite{cori1981planar}, and further developed by Schaeffer~\cite{Schaeffer:these} and by Bouttier, Di Francesco and Guitter~\cite{bouttier2002census,BouttierFG04}. Bijective methods for planar maps have gained a high level of generality over the years, and one can essentially distinguish three main methods:
\begin{enumerate}
\item[(i):] Bijections that are obtained by showing that each map from the considered family can be endowed with a canonical orientation (or, more generally, a canonical weighted biorientation), typically with properties of minimality (no clockwise cycle), accessibility 
(from any vertex one can reach the root-face or root-vertex), and specification of vertex outdegrees and face degrees;  
one can then specialize a general bijection (keeping track of vertex outdegrees and face degrees) that encodes a minimal accessible orientation by a decorated plane tree. Two different such general bijections were provided respectively in~\cite{Bernardi-Fusy:dangulations} and in~\cite{albenque2013generic}. These gave a unified setting encompassing most bijections that were known for planar map families,  
 and making it possible to obtain extensions such as the enumeration of planar maps of fixed girth with control on the face-degrees~\cite{OB-EF:girth}.    
\item[(ii):] Slice decompositions~\cite{bouttier2012planar}, where the considered maps are cut along certain geodesic paths, yielding a decomposition into slice-maps, and  slice-maps are themselves recursively decomposed, again by cutting along certain geodesic paths. As (i) this method is well suited for the enumeration of maps with control on the girth and face-degrees~\cite{BouttierG14,bouttier2024}, with extension to the so-called irreducible setting.      
\item[(iii):] A recent bijective construction~\cite{duchi2025} turns (via certain rewiring operations) the description trees given by a generic  order-one catalytic equation into trees that are specified by a context-free grammar. This establishes a bridge from Tutte's method to a bijective encoding of maps by  trees that are easy to count.   
\end{enumerate}

In this paper we consider the class of so-called \emph{5c-triangulations}, i.e., triangulations of the 5-gon with no separating triangle nor separating quadrangle\footnote{In a planar map a cycle is called \emph{separating} if there is at least one vertex inside and at least one vertex outside (the underlying graph is thus disconnected by deletion of the cycle).}. These have been recently considered in~\cite[Sec.12]{OB-EF-SL:Grand-Schnyder},\cite{BernardiFL23}, where it is proved that these maps are characterized by the existence of Schnyder-type combinatorial structures that admit several incarnations: as  corner labelings, as certain woods on 5 trees, or as outdegree-constrained orientations on the primal-dual map (in~\cite{BernardiFL23} it is also shown that the wood incarnation yields a face-counting crossing-free straight-line drawing algorithm). These maps are also closely related to 5-connected planar triangulations, which are known to be the triangulations with no separating triangle nor separating quadrangle. The enumeration of rooted 5-connected planar triangulations
has been solved by Gao, Wanless and Wormald~\cite{gao2001} via a bivariate application of the substitution method, yielding an explicit algebraic expression for the generating function. No bijection by the above mentioned methods was known for 5-connected planar triangulations, whereas methods (i) and (ii) are known to apply~\cite{Fu07b,BouttierG14} to triangulations of the 4-gon with no separating triangle ---called irreducible--- which are closely related to 4-connected planar triangulations.  

Our main contribution is to exhibit a simple family of trees that are in bijection to 5c-triangulations. 
A \emph{plane tree} is an unrooted tree embedded in the plane. The non-leaf vertices are called \emph{nodes}, the edges incident to a leaf are called \emph{legs}, while the other ones are called \emph{inner edges}. A \emph{$k$-regular} plane tree is a plane tree such that all nodes have degree $k$. It is called \emph{leg-balanced} if for every inner edge, one of its two half-edges is followed by a leg in clockwise order around the incident node, while the other half-edge is followed by an inner edge. The following is our main result (see Figure~\ref{fig:example_main_bijection}): 

\fig{width=12cm}{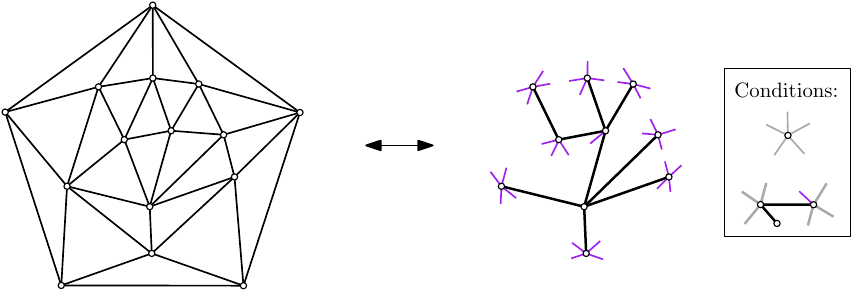}{Left: a 5c-triangulation. Right: the corresponding leg-balanced 5-regular plane tree.}

\begin{theo}\label{theo:main}
There is an explicit bijection $\phi$ between 5c-triangulations with $n\!+\!5$ vertices
and leg-balanced 5-regular plane trees with $n$ nodes, for $n\geq 1$.
\end{theo}

Let us note that bijective constructions are already known to relate $k$-connected triangulations to (sub-)families of $k$-regular plane trees for $k\in\{3,4\}$. Precisely, 3-connected (i.e., simple) plane triangulations with $n+3$ vertices are in bijection with $3$-regular plane trees ---with nodes properly bicolored--- having $n$ white nodes and such that all leaves are adjacent to black nodes~\cite{FuPoScL,Bernardi-Fusy:dangulations}. And triangular dissections of the 4-gon with no separating triangle (called irreducible triangulations of the 4-gon, these are closely related to 4-connected triangulations) having $n+4$ vertices are in bijection with $4$-regular plane trees having $n$ nodes~\cite{Fu07b}. 
The bijection of Theorem~\ref{theo:main} thus extends this pattern to $k=5$.    

\medskip

\noindent{\bf Outline.} Our bijection $\phi$ follows the scheme (i) and is close to the bijection of~\cite{FuPoScL} for irreducible quadrangulations of the hexagon, which relies 
 on a certain biorientation of the inner edges, called here a \emph{tree-biorientation}.  Precisely it is shown in~\cite{FuPoScL} that irreducible quadrangulations of the hexagon are characterized by the existence of a (unique) tree-biorientation with all inner vertices of outdegree $3$.    
We give in Section~\ref{sec:gen_bij} an easy bijection $\Phi$   
between tree-biorientations on quadrangular dissections and bicolored plane trees; the construction consists (as in~\cite{FuPoScL}) in deleting the ingoing half-edges and the outer vertices/edges, while the inverse mapping $\Psi$ relies on certain closure operations, each one  creating a quadrangular inner face.  
We then show in Section~\ref{sec:5c_biori} that the angular maps of 5c-triangulations ---certain quadrangulations of the 10-gon--- are characterized by the existence, and uniqueness, of a tree-biorientation with outdegree~$5$ at inner white vertices and outdegree $2$ at inner black vertices (these tree-biorientations are closely related to the 5c-orientations of~\cite{BernardiFL23}, see Remark~\ref{rk:transf}). It then remains to characterize the corresponding plane trees. A difficulty here compared to~\cite{FuPoScL} is that we not only need to control vertex outdegrees (which is easily tracked by the mapping $\Phi$), but also some vertex degrees, precisely 
the fact that black vertices have degree 3. 
We show in Section~\ref{sec:special} ---using lemmas established in Section~\ref{sec:partial}---  that there is sufficient partial control on black vertex degrees via $\Phi$ to track this property, and characterize the  corresponding trees by local properties only: upon simplication of the tree-shape, these trees identify to 5-regular leg-balanced plane trees. 
In Section~\ref{sec:counting} we obtain from the bijection $\phi$ an expression for the generating function of rooted 5c-triangulations (and its derivative). We can then also express the generating function of rooted 5-connected triangulations with root-vertex degree $5$, since the two map families  are related by a simple vertex-shelling extraction process. By singularity analysis we also obtain asymptotic estimates for the counting coefficients of both families.
Finally, in Section~\ref{sec:algo} the bijection $\phi$ is applied to the random generation and entropy encoding of 5-connected triangulations. We conclude with open questions on bijective extensions, in particular regarding more general counting results that can be obtained from the substitution approach of~\cite{gao2001}.

\section{Definitions on maps and orientations}
A \emph{plane map} (shortly hereafter, a map) is a connected multigraph embedded in the plane without edge-crossings, considered up to continuous deformation\footnote{For connected multigraphs embedded on the sphere the terminology of planar maps is commonly used; upon projection a plane map is equivalent to a planar map with a distinguished face.}. The unbounded face is called \emph{outer face}, while the other faces are called \emph{inner faces}. A \emph{corner} is the angular sector between two consecutive half-edges around a vertex. The \emph{degree} of a vertex or face is its number of incident corners. The \emph{outer degree} is the degree of the outer face. 
Edges, vertices, and corners are called \emph{outer} or \emph{inner} whether they are incident to the outer face or not; an \emph{inner half-edge} is an half-edge on an inner edge. A map is \emph{rooted} by marking a corner, called the \emph{root}, that is incident to the outer face.
A \emph{dissection} is a plane map whose outer face contour is a simple cycle. A triangular (resp. quadrangular) dissection is a dissection whose inner faces have degree $3$ (resp. degree $4$). Note that a quadrangular dissection has even outer degree and is bipartite (since all faces have even degree), hence the vertices can be colored black or white so that adjacent vertices have different colors; a dissection has two such colorings, which differ by a color switch.  
A quadrangular dissection is \emph{bicolored} if it is endowed with such a vertex bicoloration. 

A \emph{biorientation} of a dissection is an orientation of the inner half-edges. The inner edges with $i\in\{0,1,2\}$ half-edges outgoing are called $i$-way. In case all inner edges are 1-way, the bi-orientation is simply called an orientation. The \emph{outdegree} (resp. \emph{indegree}) of a vertex is the number of incident outgoing (resp. incoming) half-edges. 
A biorientation is called \emph{admissible} if every outer vertex has outdegree $0$. In a biorientation a \emph{directed path} from vertex $v$ to vertex $v'$ is a path $v=v_0,e_0,v_1,\ldots,e_k,v_{k+1}=v'$ such that for $0\leq i\leq k$ the edge $e_i$ is either a 2-way edge connecting $v_i$ and $v_{i+1}$, or a 1-way edge from $v_i$ to $v_{i+1}$. It is called \emph{simple} if all vertices along the path are distinct, except possibly $v=v'$ in which case it is called a simple directed closed path.    
 A \emph{clockwise circuit} is a simple directed closed path with the interior region on its right. A biorientation is called \emph{minimal} if it is admissible and has no clockwise circuit. 
 It is called \emph{accessible} if from any inner vertex $v$ there is a directed path from $v$ to some outer vertex.  
 A \emph{tree-biorientation} of a dissection $M$ is a minimal bi-orientation of $M$ whose 2-way edges form a tree spanning the inner vertices. Note that a tree-biorientation is accessible, and moreover for every 1-way edge $e$ connecting two inner vertices, the unique cycle $C_e$ formed by $e$ and 2-way edges is such that $e$ has the interior of $C_e$ on its left.

\section{Quadrangular tree-biorientations and bicolored plane trees}
\subsection{The general bijection}\label{sec:gen_bij}
For $d\geq 1$ we let $\cQ_d$ be the family of tree-biorientations on bicolored quadrangular dissections of outer degree $2d$. We define a \emph{bicolored plane tree} as a plane tree with the following conditions:
\begin{itemize}
\item
Some leaves are marked, their incident edge is called a \emph{leg},
\item
All other vertices are called \emph{nodes}, an edge connecting two nodes is called an \emph{inner edge},
\item
Nodes are colored black or white, such that adjacent nodes are of different colors.
\end{itemize}
The \emph{excess} of a bicolored plane tree is the number of legs minus the number of inner edges.  
For $d\geq 1$, let $\cT_d$
 be the family of bicolored plane trees of excess $d$.   
For $Y\in\cQ_d$, with $Q$ the underlying bicolored quadrangular dissection, let $T=\Phi(Y)$ be the bicolored plane tree obtained by deleting the outer vertices and outer edges of $Q$, and all incoming half-edges, so that the 1-way edges become the legs of the obtained tree, see the left side of Figure~\ref{fig:Phi} for an example. By the Euler relation, $Q$ has $2n+d-2$ inner edges, among which $n-1$ are 2-way and $n+d-1$ are 1-way in $Y$. In the obtained tree $T$, the inner edges (resp. legs) correspond to the 2-way (resp. 1-way) edges of $Y$, hence $T\in\cT_d$.  

\fig{width=14cm}{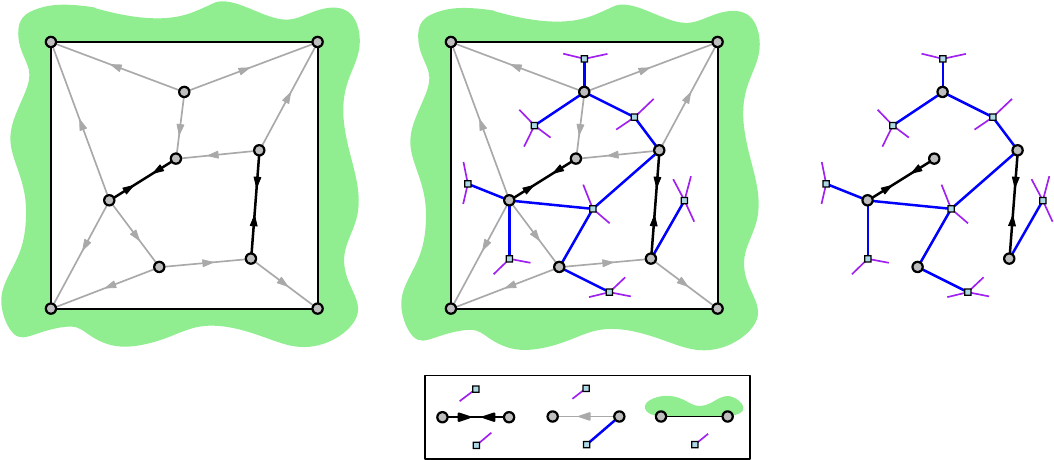}{Left: a biorientation $O\in\cO_4$. Right: the corresponding bimobile $T=\Phi_-(O)$.}

\fig{width=14cm}{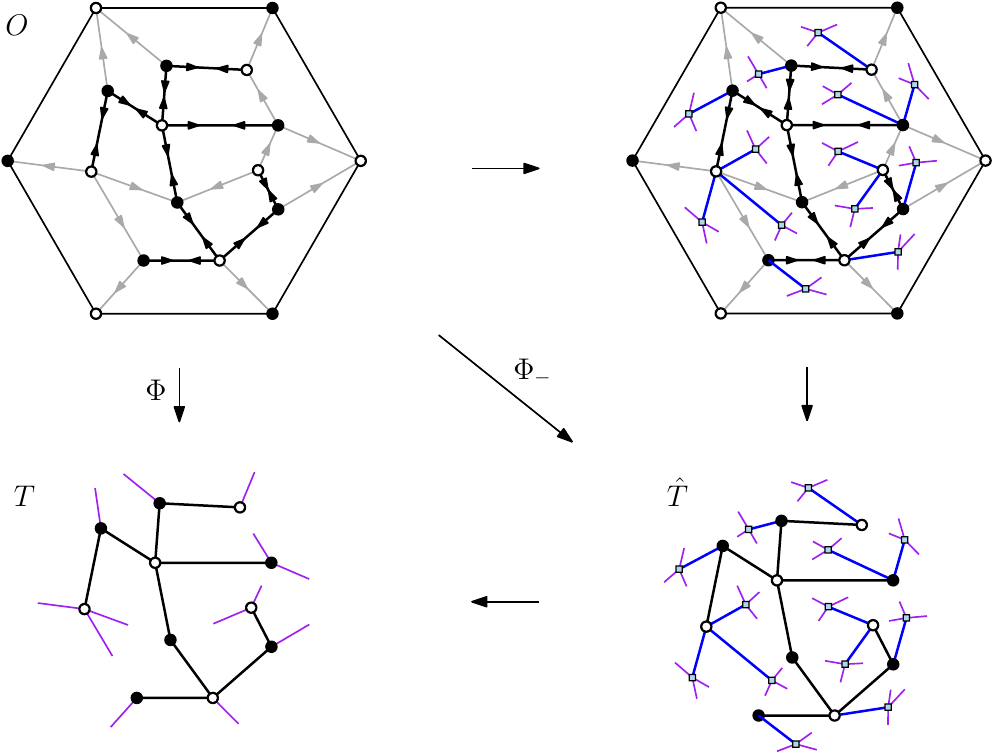}{Top-left: a tree-biorientation $Y\in\cQ_3$. Bottom-left: the corresponding bicolored plane tree $T=\Phi(Y)\in\cT_3$, which can also be obtained (right column)  from the bimobile $\hat{T}=\Phi_-(Y)$ upon deleting the 3 legs at every square vertex, and turning the unique incident inner edge into a leg.}

\begin{prop}\label{prop:Phi}
For $d\geq 1$, the mapping $\Phi$ is a bijection from $\cQ_d$ to $\cT_d$. 
For $Y\in\cQ_d$, with $Q$ the underlying bicolored quadrangular dissection, every black (resp. white) inner vertex $v$ of $Q$ corresponds to a black (resp. white) node in $T=\Phi(Y)$, and the outdegree of $v$ corresponds to the degree of the node.  
\end{prop}
\begin{proof}
We first recall the general master bijection $\Phi_-$ of~\cite{Bernardi-Fusy:dangulations}  for minimal accessible biorientations, illustrated in Figure~\ref{fig:Phi_minus}, and then will explain how it specializes to tree-biorientations.  
For $d\geq 1$, let $\cO_d$ be the family of dissections of outer degree $d$ endowed with a minimal accessible biorientation with no 0-way edge. 
We define a bimobile as a plane tree with the following conditions:
\begin{itemize}
\item
Some leaves are marked, their incident edge is called a \emph{leg},
\item
All other vertices are called \emph{nodes}, an edge connecting two nodes is called an \emph{inner edge},
\item
Nodes are of two types, round or square, such that every leg is incident to a square node,
and there is no square-square inner edge. 
\end{itemize}
The \emph{excess} of a bimobile is the number of legs, minus the number of round-square inner edges, minus twice the number of round-round inner edges. 
For $d\geq 1$ let $\cB_d$ be the family of bimobiles of 
 excess $d$. For $O\in\cO_d$, whose vertices are considered round, we let $T=\Phi_-(O)$ be obtained by the following operations:
\begin{itemize}
\item
insert a square vertex $s_f$ in each inner face $f$;
\item
for each half-edge $h=(v,e)$ of $O$ such that the opposite half-edge is outgoing, with $f$ the face on the left of $h$, insert a leg from $s_f$ pointing (but not reaching) toward $v$;
\item
for each half-edge $h=(v,e)$ on an outer edge and having an inner face $f$ on its left, insert a leg from $s_f$ pointing (but not reaching) toward $v$;
\item
for each 1-way edge $e=(u,v)$, with $f$ the face on the left of $e$, insert an edge connecting $s_f$ to $u$;
\item
delete the outer vertices and all edges of $O$ apart from the 2-way edges. 
\end{itemize}
It is shown in~\cite{Bernardi-Fusy:dangulations} that for each $d\geq 1$ the mapping $\Phi_-$ is a bijection from $\cO_d$ to $\cB_d$. Several parameter correspondences hold. In particular:
\begin{itemize}
    \item
    every round vertex of degree $k$ in $T$ corresponds to an inner vertex of outdegree $k$ in $O$,
    \item
 every square vertex of degree $k$ in $T$ corresponds to an inner face of degree $k$ in $O$,
\item
  every round-round edge of $T$  corresponds to a 2-way edge of $O$, with the same extremities in $T$ as in $O$.   
\end{itemize}

\medskip

We now consider the specialization of $\Phi_-$ to $\cQ_d$. For $Y\in\cQ_d$, 
with $Q$ the underlying bicolored quadrangular dissection, let $\hat{T}=\Phi_-(Y)$. 
We endow the round nodes of $\hat{T}$ with the bicoloration inherited from the vertex bicoloration of $Q$; note that two adjacent nodes of $\hat{T}$ must have different
 colors (since the corresponding vertices are also adjacent in $Q$). Moreover, since the round-round edges form a subtree of $\hat{T}$ spanning all round vertices, and since there are no square-square edges in $\hat{T}$, the square vertices of $\hat{T}$ have a single round neighbour, plus 3 incident legs (since the corresponding inner face of $Q$ has   degree $4$). The \emph{reduction} of $\hat{T}$ is the bicolored plane tree $T$ obtained by replacing every square vertex of $\hat{T}$ by a marked leaf, so that the incident inner edge becomes a leg. The fact that $\hat{T}$ has excess $d$ transfers to the fact that $T\in\cT_d$. 
 Moreover it is easy to see that $T=\Phi(Y)$, as illustrated in Figure~\ref{fig:Phi}.  Conversely, for $T\in \cT_d$, we may expand every marked leaf into a square vertex with 3 attached legs, yielding a bimobile $\hat{T}$ of excess $d$. Letting $Y\in\cO_d$ be such that 
 $\hat{T}=\Phi_-(Y)$, and $Q$ the underlying dissection, we clearly have that $Y$ is a tree-biorientation of $Q$, that all inner faces of $Q$ have degree $4$, and that the vertex bicoloration inherited from $T$ (and extended to the outer vertices) is a proper bicoloration of the vertices of $Q$, since the 2-way edges form a tree spanning the inner vertices of $Q$. 
\end{proof}

\begin{remark}\label{rk:diss_hexagon}
As a specialization of Proposition~\ref{prop:Phi}, it is shown in~\cite{FuPoScL} that quadrangular dissections of the hexagon can be endowed with a tree-biorientation where all inner vertices have outdegree $3$ iff the dissection is irreducible, i.e., every 4-cycle bounds a face. Moreover, in that case such a tree-biorientation is unique. 
Thus Proposition~\ref{prop:Phi} gives a bijection between such dissections and unrooted bicolored binary trees, which is the bijection of~\cite{FuPoScL}. \hfill$\triangle$
\end{remark}

\fig{width=12cm}{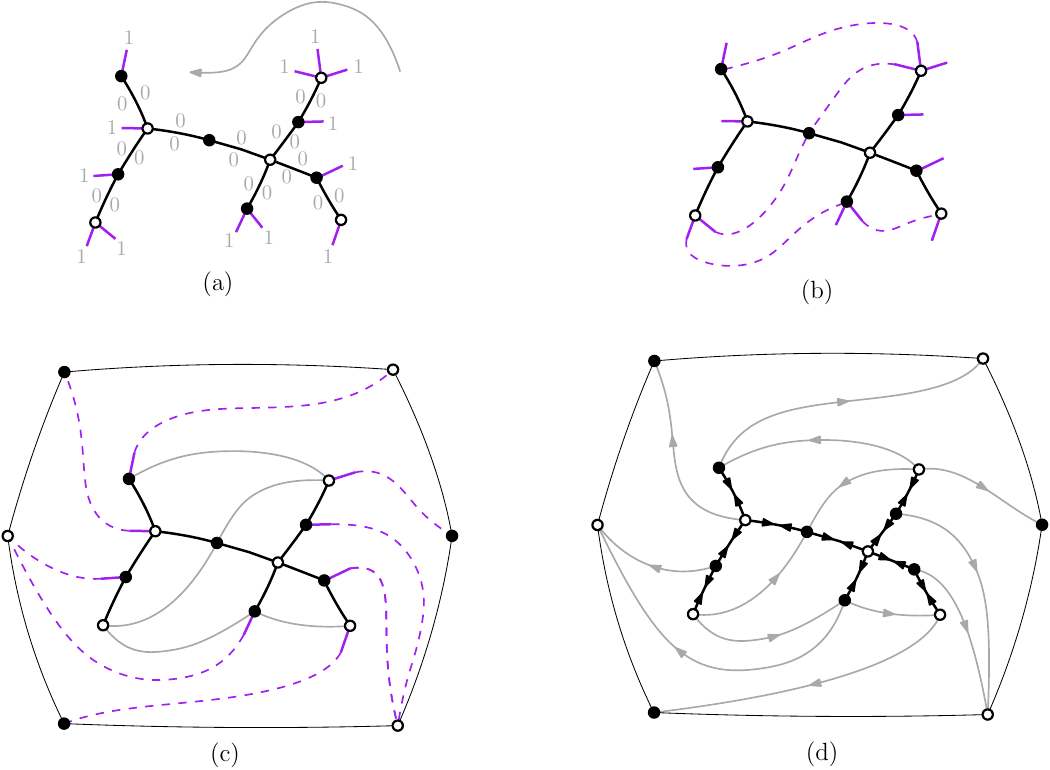}{The closure mapping $\psi$, from a bicolored plane tree $T\in\cT_3$ to a tree-oriented quadrangular dissection $Y\in\cQ_3$.}

\begin{remark}
The inverse $\psi$ of $\phi$ has an explicit description (illustrated in Figure~\ref{fig:bijection_psi}) that is induced by the mapping 
$\Psi_-=(\Phi_-)^{-1}$ described in~\cite{Bernardi-Fusy:dangulations}, and is the natural extension of the closure mapping described in~\cite{FuPoScL} for  bicolored binary trees.   
For $T\in\cT_d$, we consider a counterclockwise tour around $T$ (i.e., with the outer face on our right), yielding a cyclic word with $1$'s and $0$'s, where we write $1$ when passing along a marked leaf (i.e., a leaf incident to a leg), and write $0$ when passing along an edge. 
Whenever we see the pattern $1000$, meaning a leg followed by three edges, we perform a so-called \emph{local closure} operation, which consists in extending the end of the leg to reach the end of the third edge, so as to close a quadrangular face. The edge resulting from the local closure is not anymore considered as a leg, and accordingly the contour word for the outer face is updated by replacing the pattern $1000$ by~$0$. Local closure operations are performed until having no pattern $1000$ in the contour word. At that point we create an enclosing $2d$-gon, and perform final closure operations that match the ends of the remaining legs with vertices of the $2d$-gon so as to form quadrangular inner faces only. The resulting quadrangular dissection of outer degree $2d$ is naturally endowed with a tree-orientation, where the edges originating from $T$ are 2-way, while the edges originating from legs are  1-way, outward of the original incident node. 
The image of $T$ by $\psi$ is the obtained tree-oriented quadrangular dissection $Y$ (note that by construction we have $T=\phi(Y)$).      \hfill$\triangle$
\end{remark}

\subsection{Two lemmas giving partial control on black vertex degrees}\label{sec:partial}
The parameter-correspondence in Proposition~\ref{prop:Phi} only allows control on the outdegrees of inner vertices, not on their indegrees. As a first remark, there is a situation where we have a simple control on black vertex indegrees: if all legs are at black nodes of the bicolored tree, then the black vertices have indegree $0$ (because all 1-way edges are directed from black to white). For instance, in the bijection of Remark~\ref{rk:diss_hexagon}, if we specialize to the case where all legs are at black vertices, then we get a bijection to quadrangular irreducible dissections of the hexagon such that all black inner (resp. outer) vertices have degree $3$ (resp. $2$), which are themselves in easy bijection to simple planar triangulations. This specialization is given in~\cite{FuPoScL} and is the bijection for simple triangulations (also obtained in~\cite{Bernardi-Fusy:dangulations}) that is mentioned at the end of the introduction. 

However, in order to apply Proposition~\ref{prop:Phi} to the tree-biorientations for  5c-triangulations (to be described in Section~\ref{sec:5c_biori}), a difficulty is that we will need control on black vertex degrees (to be all equal to $3$) while having legs at white nodes in the associated bicolored plane trees. 
This control will be made possible by the two lemmas given below. In a bicolored plane tree $T$ a \emph{node-corner} (resp. \emph{leaf-corner}) is a corner incident to a node (resp. a leaf). The \emph{leaf-index} of a node-corner $c$ is the number of node-corners (excluding $c$) between $c$ and the next leaf-corner in a clockwise walk around the tree (with the outer face on the left of the walker). In $Q=\Psi(T)$ a node-corner $c$ corresponds to what we call a \emph{sector}, i.e., an interval of corners around an inner vertex, delimited by two outgoing half-edges, and such that the non-extremal half-edges in the sector are incoming (thus an inner vertex of outdegree $k$ has $k$ sectors). We say that the sector is empty if it has no incoming half-edge, i.e., it is a corner between two outgoing half-edges around an inner vertex.  

\begin{lemma}\label{lem:legindex1}
Let $T\in\cT_d$ and  let $Q=\Psi(T)$. For $c$ a node-corner of $T$, if $c$ has leg-index smaller than $3$, then the corresponding sector in $Q$ is empty. 
\end{lemma}
\begin{proof}
Let $\ell$ be the next leg encountered after $c$ in a clockwise walk around $T$. Since the leg-index of $c$ is smaller than $3$, then when $\ell$ is closed it has to pass over $c$, thus $c$ is a corner in $Q$, in the inner face resulting from the closure of $\ell$. 
\end{proof}

\begin{lemma}\label{lem:legindex2}
Let $T\in\cT_d$ be such that there is no pair of legs that are consecutive around a black node, and let $Y=\Psi(T)$, with $Q$ the underlying dissection. For $c$ a corner at a black node of $T$, if $c$ has leg-index at least $3$ then the corresponding sector in $Q$ is not empty. Moreover, every outer black vertex has positive indegree. 
\end{lemma}
\begin{proof}
Assume the sector corresponding to $c$ is empty. Then $c$ is a corner of $Q$. Let $f$ be the inner face of $Q$ containing $c$. Let $v_1,v_2,v_3,v_4$ be the four vertices   
in counterclockwise order around $f$, with $v_1$ the vertex at $c$. Note that the edge $\{v_1,v_2\}$ is a 2-way edge since $c$ has non-zero leg-index, and $\{v_1,v_4\}$ is either 2-way or 1-way from $v_1$ to $v_4$. Hence, the unique 1-way edge having $f$ on its left is either $\{v_2,v_3\}$ or $\{v_3,v_4\}$. It can not be $\{v_2,v_3\}$, as this would imply that the leg-index of $c$ is equal to $1$. Hence it has to be $\{v_3,v_4\}$ (directed from $v_3$ to $v_4$). In that case $(v_2,v_3)$ can not be 2-way, as this would imply that the leg-index of $c$ is equal to $2$. Hence $\{v_2,v_3\}$ is 1-way, from $v_3$ to $v_2$. But then we have the contradiction that $v_3$ has two consecutive incident edges that are outgoing 1-way, and these give a pair of legs that are consecutive at the black node corresponding to $v_3$ in $T$. 

Regarding the second statement, assume that there is a black outer vertex $b$ with no incident incoming edge, i.e., no incident inner edge. Let $f$ be the unique inner face incident to $b$. Let $b,w,b',w'$ be the vertices in counterclockwise order around $f$. The edges $\{b,w\}$ and $\{b,w'\}$ are outer edges, so that $w,w'$ are outer vertices. Note that $Q$ can not have an inner edge connecting two outer vertices; indeed such an edge would have to be incoming at both extremities, which is not possible since we forbid 0-way edges in tree-biorientations. This implies that $b'$ is an inner vertex. 
Thus the two inner edges $\{b',w\}$ and $\{b',w'\}$ are 1-way out of $b'$, and they are consecutive around $b'$. We get the contradiction that these two edges yield two consecutive legs at a black node of $T$.  
\end{proof}

\section{Application to 5c-triangulations}
\subsection{Tree-biorientations for 5c-triangulations}\label{sec:5c_biori}
For $M$ a triangular dissection (whose vertices are colored white), one obtains a quadrangular dissection $Q=\iota(M)$, called the \emph{angular map} of $M$, by the following classical procedure: a black vertex is inserted inside each inner face $f$, and connected by 3 new edges to the white vertices around $f$; then the edges of $M$ are deleted. For $Q$ to be a dissection (i.e., have simple outer contour) requires the mild assumption that $M$ has no inner face incident to two outer edges of $M$. We let $\cF$ be the family of triangular dissections of the 5-gon with no inner face incident to two outer edges, and let $\cG$ be the family of quadrangular dissections of outer degree $10$ with all black vertices of degree $3$; note that $\iota$ is a bijection between $\cF$ and $\cG$. Let $\cFfc\subset\cF$ be the family of 5c-triangulations, and let $\cGfc\subset\cG$ be the image of $\cFfc$ under $\iota$.

For $Q\in\cG$, a tree-biorientation of $Q$ is called a \emph{5c-biorientation} if every inner white vertex has outdegree $5$ and every inner black vertex has outdegree $2$. 

\fig{width=\linewidth}{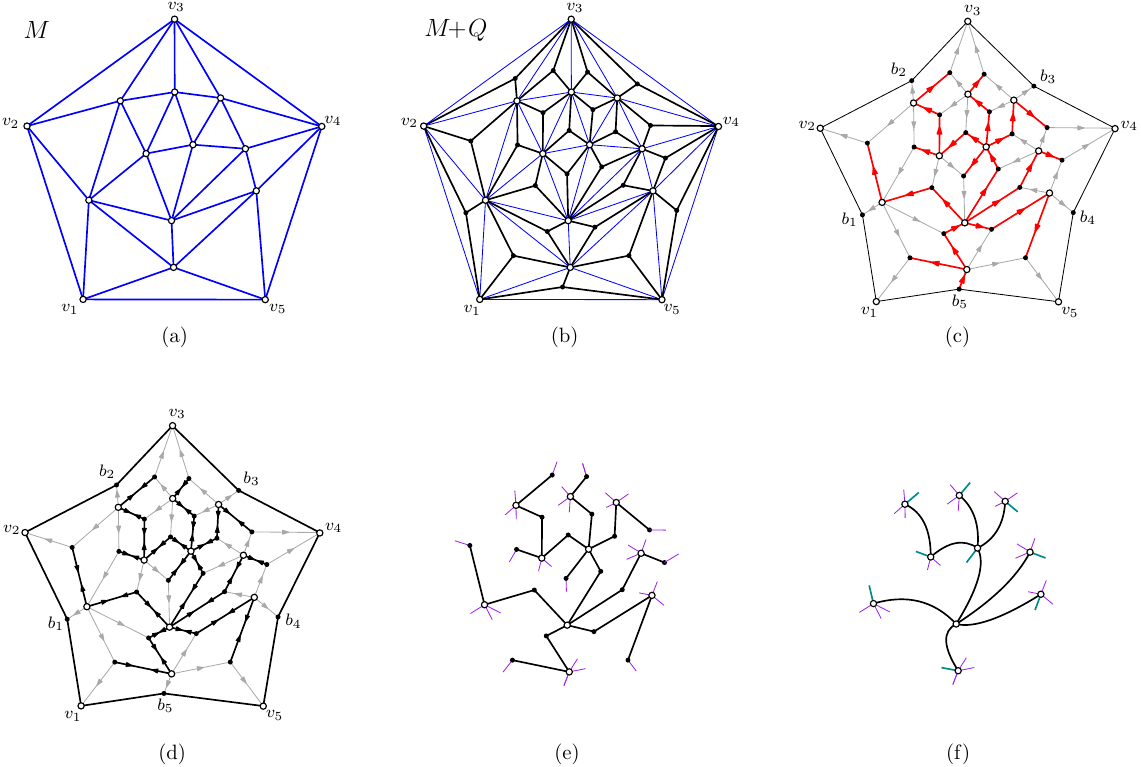}{(a) The 5c-triangulation $M$ of Figure~\ref{fig:example_main_bijection}, (b) superimposed with $Q=\iota(M)$; (c) $Q$ is endowed with its  minimal $b_5$-rooted regular orientation $X_0$, with the left co-accessibility spanning tree $\tau_0$ colored red; (d) making all edges of $\tau_0$ 2-way and reversing the inner edge incident to $b_5$ yields the 5c-biorientation $Y_0$ of $Q$;
 (e) the tree $T=\Phi(Y_0)$; (f) the leg-balanced 5-regular plane tree $\tau=\phi(M)$ is obtained as the reduction of $T$.}

\begin{lemma}\label{lem:5cbiori}
Let $Q\in\cG$. Then $Q$ admits a 5c-biorientation iff $Q\in\cGfc$. 
In that case $Q$ has a unique 5c-biorientation.  
\end{lemma}
\begin{proof}
Let $Q\in\cGfc$ and let $M$ be the 5c-triangulation such that $Q=\iota(M)$. We now describe the process to compute the 5c-biorientation of $Q$, illustrated in Figure~\ref{fig:ori_5c_bis}(c)-(d).   
A \emph{regular orientation} of $Q$ is an orientation of the inner edges such that all outer vertices have outdegree~$0$ except one, of outdegree $1$ and called the \emph{root-vertex},  and the inner white (resp. black) vertices have outdegree $4$ (resp. $1$). As shown in~\cite[Sec.12]{OB-EF-SL:Grand-Schnyder} (see also Section A.3 and Figure~13(b)-(c) in~\cite{BernardiFL23}),  the fact that $M$ has no separating 3-cycle implies that $Q$ can be endowed  with a regular orientation $X$, and furthermore the fact that $M$ has no separating 4-cycle ensures that any regular orientation of $M$ is \emph{co-accessible} with respect to its root-vertex~$v$, meaning that for every vertex $w$ there is a directed path from $v$ to $w$. 
Note that if $v\neq b_5$, we can reverse a directed path from $v$ to $b_5$ and obtain a regular orientation rooted at $b_5$. According to~\cite{Felsner:lattice}, one can then consider the minimal $b_5$-rooted regular orientation $X_0$ of $M$. A \emph{co-accessiblity tree} is then defined as a subtree $\tau$ of $Q$ spanning the vertex-set $V$ made of $b_5$ and the inner vertices of $Q$, and such that the orientation on $\tau$ given by $X_0$ is the orientation away from the root $b_5$. In that case, an \emph{external edge} is an inner edge $e$ of $Q$ not in $\tau$ and connecting two vertices of $V$; letting $\gamma_e$ be the unique cycle formed by $e$ and edges of $\tau$, the external edge $e$ is called \emph{counterclockwise} (resp. clockwise) if it has the interior of $\gamma_e$ on its left (resp. on its right). 
It follows from~\cite{Bernardi07} that there is a unique co-accessibility tree $\tau_0$ such that every external edge is counterclockwise; it is called the \emph{left co-accessibility tree}.  If we make every edge of $\tau_0$ 2-way, then we obtain a biorientation $Y_0'$ of $Q$ such that the 2-way edges form a tree spanning $V$. Note that every inner vertex has clearly increased its outdegree by $1$, so that white (resp. black) inner vertices
 have outdegree $5$ (resp. $2$). Moreover, the bi-orientation $Y_0'$ is minimal; 
 indeed $\tau_0$ is a left co-accessibility tree of $Y_0'$, upon interpreting every 2-way edge as a double edge forming a counterclockwise cycle; and as shown in~\cite{Bernardi07} the existence of a left co-accessibility tree guarantees minimality (no. clockwise cycle). Upon turning the 2-way edge incident to $b_5$ into a 1-way edge directed toward $b_5$, we thus obtain a 5c-biorientation $Y_0$ of $Q$.  
  
To prove uniqueness, assume for contradiction that $Q$ has a 5c-biorientation $Y_1\neq Y_0$. 
Let $\tau_1$ be the tree formed by the 2-way edges and the inner edge incident to $b_5$. Note that $\tau_1\neq\tau_0$ (otherwise one would have $Y_1=Y_0$). If we turn the edges in $\tau_1$ into 1-way edges directed away from the root $b_5$, then we obtain a regular orientation $X_1$ of $Q$ rooted at $b_5$, and for which $\tau_1$ is a left co-accessibility tree. By existence of a left co-accessibility tree, the orientation $X_1$ has to be minimal, hence it is the minimal $b_5$-rooted regular orientation of $Q$ (hence $X_0=X_1$), for which $\tau_1$ is the unique left co-accessibility tree. We reach the contradiction that $\tau_1=\tau_0$. 

Finally, let $Q\in\cG\backslash\cGfc$ and assume that $Q$ can be endowed with a 5c-biorientation $X$. Let $M$ be the triangulation of the 5-gon such that $Q=\iota(M)$.  
Since $Q\notin\cGfc$, there is a separating 4-cycle $\gamma=(w_1,w_2,w_3,w_4)$ in $M$. Let $f_i$ be the inner face of $M$ inside $\gamma$ and incident to the edge $\{w_i,w_{i+1}\}$, and let $b_i$ be the corresponding black vertex of $Q$. 
Consider the 8-cycle $\gamma_Q=(w_1,b_1,\ldots,w_4,b_4)$ of $Q$. Inside $\gamma_Q$, let $n\geq 1$ be the number of white vertices, so that (by the Euler relation) there are $2n-2$ black vertices, and $6n-2$ edges of $Q$. 
Let $S$ be the total outdegree of vertices inside $\gamma_Q$. The outdegree condition gives $S=5n+2(2n-2)=9n-4$. On the other hand, note that $S$ is at most the number of edges inside $\gamma_Q$ plus the number $k$ of 2-way edges with both extremities inside $\gamma_Q$. Since the 2-way edges form a tree, we have $k\leq 3n-3$. Hence, $S\leq 9n-5$, giving a  contradiction.     
\end{proof}

\begin{remark}\label{rk:transf}
It would also be possible, though a bit less direct, to obtain the 5c-biorientation of $Q=\iota(M)$ from the minimal 5c-orientation~\cite{BernardiFL23} on the primal-dual completion of $M$, via transfer rules similar to those used in~\cite{FuPoScL}. To prove that the transferred biorientation is minimal, one would have to resort to an analogue of Lemma~8.13 in~\cite{FuPoScL}, for which one would have to show that 5c-orientations are accessible. \hfill$\triangle$
\end{remark}

\subsection{Specialization of $\Phi$ to 5c-biorientations}\label{sec:special}

Let $\wcQfc$ be the subfamily of $\cQ_{5}$ such that all inner white vertices have outdegree $5$ (resp. $2$). Let $\cQfc$ be the subfamily of $\wcQfc$ where all black vertices (including the outer ones) have degree $3$. Note that the family $\cQfc$ is in bijection with   5c-triangulations, by Lemma~\ref{lem:5cbiori} combined with the bijection $\iota$. Let $\wcTfc$ be the subfamily of $\cT_{5}$ such that all white (resp. black) nodes have degree $5$ (resp. $2$).  Note that there are two types of black nodes in $T\in\wcTfc$: those incident to one inner edge and one leg (type I), and those incident to two inner edges (type II).  The fact that the excess is $5$ translates to the fact that there are as many black nodes of both types. 
Let $\cTfc$ be the subfamily of $\wcTfc$ where the following two conditions hold:
\begin{itemize}
\item[{\bf C1}:]
For each black node $b$ of type I, let $w$ be the adjacent white node, and $e$ the next edge after $(w,b)$ in clockwise order around $w$; then $e$ leads to a black node of type II.  
\item[{\bf C2}:]
For each black node $b$ of type II, let $w,w'$ be the adjacent white nodes, $e$ the next edge after $(w,b)$ in clockwise order around $w$, and $e'$ the next edge after $(w',b)$ in clockwise order around $w'$; then exactly one of $(e,e')$ leads to a black node of type II. 
\end{itemize} 

\begin{prop}
The mapping $\Phi$ gives a bijection between $\wcQfc$ and $\wcTfc$, which specializes into a bijection between $\cQfc$ and $\cTfc$.
\end{prop} 
\begin{proof}
The first statement is a direct consequence of the paramater correspondence in Proposition~\ref{prop:Phi}. 
Let $T\in \wcTfc$ with $n$ white nodes.
Assume that $T\in\cTfc$. Then by definition every black node must have an incident corner of leg-index at least $3$. Hence, by Lemma~\ref{lem:legindex2}, all black nodes (including the outer ones) have degree at least $3$ in $Y=\Psi(T)$. Since $T$ has $2n-2$ black nodes ($n-1$ of each type) and $n$ white nodes, its image $Y$ has $2n+3$ black vertices (including the outer ones), and a total of $3n+8$ vertices. By the Euler relation it has $6n+9$ edges. 
Hence all black vertices must have degree $3$, so that $Y\in\cQfc$. 

Assume now that  $T\notin \wcTfc$ and let us check that its image $Y=\Psi(T)$ is not in $\cQfc$. If there is at least one black node $b$ of type I not satisfying Condition~C1, then the corner after the unique incident leg in clockwise order around $b$ has leg-index smaller than $3$. The other incident corner (preceding the leg) clearly has leg-index $0$. Hence, by Lemma~\ref{lem:legindex1} $b$ corresponds to a black vertex of degree $2$ in $Y$, so that $Y\notin \cQfc$. 
Otherwise all black nodes of type I satisfy C1, so that there is at least one black node of type II that does not satisfy C2. 
Let $\tau$ be obtained from $T$ by deleting every black node $b$, turning the two incident edges into a single edge (an inner edge if $b$ is of type II, a leg if $b$ is of type I). And then let $\btau$ be obtained from $\tau$ by deleting all legs (and incident leaves). Each  corner of $\btau$ corresponds to a sector of $T$ that either carries at least one edge (leg or inner edge connected to a black node of type I) or is empty. Note also that inner edges incident to black nodes of type I are in different sectors; otherwise there would be a black node of type I not satisfying~C1. 
Since there are $n-1$ black nodes of type I, the number of non-empty sectors is at least $n-1$. For each edge $e$ of $\btau$, let $h,h'$ be the two half-edges of $e$, and let $c$ (resp. $c'$) be the corner of $\btau$ following $h$ (resp. $h'$) in clockwise order around the incident node. Let $m_e\in\{0,1,2\}$ be the number of corners among $c,c'$ whose corresponding sector is not empty. 
Each edge $e$ of $\btau$ corresponds to a black node $b_e$ of type II in $T$. 
Moreover Condition~C2 on $b_e$ is equivalent to having $m_e=1$. 
Hence, there is an edge $\epsilon$ of $\btau$ such that $m_\epsilon\neq 1$. 
Since the sum of $m_e$ over all $n-1$ edges $e$ of $\btau$ is at least $n-1$, it implies that there is an edge $\epsilon'$ of $\btau$ such that $m_{\epsilon'}=2$. In that case, in $T$ the two corners at $b_{\epsilon'}$ have leg-index smaller than $3$. Hence, by Lemma~\ref{lem:legindex1} the corresponding black vertex of $Y$ has degree $2$, so that $Y\notin\cQfc$. 
\end{proof}

To obtain Theorem~\ref{theo:main} it remains to observe that $\cTfc$ is in easy bijection with the family of leg-balanced 5-regular plane trees. As considered in the above proof, 
the \emph{reduction} of $T\in\cTfc$ is the tree $\tau$ obtained from $T$  
by deleting every black node, turning the two incident edges into a single edge (an inner edge of $\tau$ if the white node has type II, a leg if it has type I),
 see  Figure~\ref{fig:ori_5c_bis}(e)-(f).   

\begin{lemma}
The reduction procedure gives a bijection between trees in $\cTfc$ having $n$ white nodes,  and leg-balanced 5-regular plane trees having $n$ nodes, for $n\geq 1$. 
\end{lemma}
\begin{proof}
The reduction clearly gives a leg-balanced 5-regular plane tree. The inverse mapping proceeds as follows. Let $\tau$ be a leg-balanced 5-regular plane tree with $n$ nodes. A \emph{sector}  of $\tau$ is a corner of the tree $\btau$ obtained from $\tau$ by deleting all legs (and incident leaves). A sector is empty if it contains no leg. The condition of being leg-balanced ensures that $n-1$ among the $2n-2$ sectors of $\tau$ are empty. For each non-empty sector, the \emph{last leg} is the one at the end of the group of legs ordered in clockwise order. We obtain a tree $T\in\cTfc$ from $\tau$ by inserting a black node in the middle of every inner edge and in the middle of every last leg. The two mappings are easily checked to be inverse of one another.  
\end{proof}

To summarize, the bijection $\phi$ between 5c-triangulations and leg-balanced 5-regular plane trees proceeds as follows. Starting with $M\in\cFfc$, take $Q=\iota(M)$ and compute the 5c-biorientation $Y_0$ of $Q$, take $T=\Phi(Y_0)$, and let $\tau=\phi(M)$ be the reduction of $T$. And the inverse mapping $\psi$ proceeds as follows: starting from a leg-balanced 5-regular plane tree $\tau$, expand $\tau$ into a tree $T\in \cTfc$ (as in Figure~\ref{fig:ori_5c_bis}(f)$\to$(e)), then take $Y_0=\Psi(T)$, with $Q$ the underlying quadrangulation of the 10-gon, and let $M=\psi(\tau)$ be the 5c-triangulation such that $Q=\iota(M)$.

\section{Counting results}\label{sec:counting}
\subsection{Generating function of 5c-triangulations}

We first obtain from the bijection $\phi$ an explicit algebraic expression for the generating function of rooted 5c-triangulations, and its derivative. 

\begin{prop}\label{prop:count5c}
Let $f_n$ be the number of rooted 5c-triangulations with $n\geq 1$ inner vertices, and let $\Ffc(t)=\sum_nf_nt^n$ be the associated generating function. Then we have
\[
\sum_n (n-1)f_nt^n=5AB,\ \ \ \Ffc(t)=t\,(1+3A+2AB+A^2+AB^2)-2AB,
\]
where $A\equiv A(t),B\equiv B(t)$ are specified by the algebraic system
\begin{equation}\label{eq:AB}
A=t\,(A+A^2+2AB+3AB^2+B^4),\ \ \ B=t\,(1+2A+B+2AB+B^2+B^3).
\end{equation}
The radius of convergence $\rho\approx 0.24775$ of $\Ffc(t)$ is the larger of the two positive roots of the polynomial $4194304x^6 - 1339392x^5 + 319317x^4 - 1107616x^3 + 561984x^2 - 79104x + 1024$, and there is a positive algebraic constant $\kappa\approx 0.54851$
 such that $f_n= \frac{\kappa}{2\sqrt{\pi}}\,\rho^{-n}n^{-5/2}$. 
\end{prop}
\begin{proof}
A \emph{planted tree} is a tree that can be obtained by cutting a leg-balanced 5-regular plane trees at the middle of an inner edge, and keeping one of the two connected components. Such a tree is rooted at the dangling half-edge resulting from the cut. Let $\cA$ (resp. $\cB$) be the family of planted trees whose root half-edge is followed by an inner edge (resp. by a leg) in clockwise order around the incident vertex; and let $A\equiv A(t)$ (resp. $B\equiv B(t)$) be the associated generating function counted by nodes; note that trees counted by $A$ (resp. $B$) have a node (resp. a leaf) as their rightmost child. As illustrated in Figure~\ref{fig:decomp_AB}, a classical root decomposition of these trees ensures that $A,B$ are solution of~\eqref{eq:AB}. 

\fig{width=14cm}{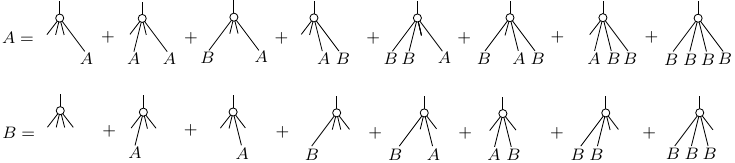}{Root decomposition of trees in $\cA$ and $\cB$.}

Note that $(n-1)f_n/5$ counts 5c-triangulations with a marked 2-way edge in the 5c-biorientation (the division by $5$ accounts for unmarking the root corner). 
By Theorem~\ref{theo:main}, it also counts leg-balanced 5-regular plane trees with $n$ nodes and a marked inner edge. The generating function for such trees is clearly $AB$, sinc cutting at the marked inner edge yields a pair of planted trees in $\cA$ and $\cB$. To obtain the expression of $\Ffc(t)$, we consider the generating function $F_1(t)$ of leg-balanced 5-regular plane trees with a marked leg followed by a leg in clockwise order around the incident node, and the generating function $F_2(t)$ of leg-balanced 5-regular plane trees with a marked inner edge. We have already seen that $F_2(t)=AB$, and a decomposition at the node incident to the marked leg ensures that $F_1(t)=t(1+3A+A^2+2AB+AB^2)$. In a leg-balanced 5-regular plane tree there are $3n+2$ legs, and $n-1$ of them are followed by an inner edge in clockwise order around the incident node, so that $2n+3$ of them are followed by a leg. 
Hence, $[t^n](F_1(t)-2F_2(t))=\frac{5}{n-1}[t^n]F_2(t)=f_n$.  

The algebraic system for $A,B$ being strongly connected and aperiodic, the Drmota-Lalley-Woods  (see~\cite[VII.6]{FSe09} and references therein) ensures that $A,B$ have same radius of convergence $\rho$ and have a square-root singular expansion in a $\Delta$-neighbourhood of $\rho$. This holds as well for $AB$. Transfer theorems of singularity analysis~\cite[VII.6]{FSe09} then ensure that $(n-1)f_n\sim \frac{c}{\sqrt{\pi}}\rho^{-n}n^{-3/2}$ for a positive constant $c$.  
Regarding explicit computations, easy algebraic manipulations (see Remark~\ref{rk:S} below) yield an algebraic equation of degree $6$ for $B$, from which one obtains the equation for $\rho$ (see~\cite[VII.7]{FSe09}). It has two positive roots, the smaller one $\approx 0.01437$ being excluded as it is smaller than the known~\cite{brown1964enumeration} radius of convergence $27/256$ for the generating function of rooted simple triangulations of the pentagon. One can also obtain an algbraic equation for the series $AB$, from which an equation for the constant $\kappa$ can be extracted (see~\cite[Theo.VII.3,Theo.VII.5]{FSe09}). We find that $\kappa$ is the smallest positive root of 
$2641807540224\,x^{12} - 9996558453964800\,x^{10} + 130110438205440000\,x^8 - 338664164994000000\,x^6 - 1765321451082421875\,x^4 - 1274277847500000000\,x^2 + 551368000000000000$.  
\end{proof}
The series expansion starts as $\Ffc(t)=t+t^6+5\,t^7+ 20\,t^8+75\,t^9+ 270\,t^{10}+ 956\,t^{11}+3365\,t^{12}+11830\,t^{13}+41665\,t^{14}+\cdots$.

\begin{remark}\label{rk:S}
The second equation in~\eqref{eq:AB} is linear in $A$, so that $A$ has a rational expression $A=g(B,t)$ in terms of $B$ and $t$. Replacing $A$ by $g(B,t)$ in the first equation of~\eqref{eq:AB} we obtain the following algebraic equation of degree $6$ for $B$:
\[
- \left( B+1 \right) ^{6}{t}^{2}+2\, \left( 3\,{B}^{2}+1 \right) 
 \left( B+1 \right) ^{2}t-{B}^{2}-2\,B=0.
\] 
We also have $\Ffc(t)$ rational in terms of $t$ and $B$ (upon replacing $A$ by $g(B,t)$ in the above expression of $\Ffc(t)$). Defining $S\equiv S(t)$ as $S=-B/(1+B)$ (so that $B=-S/(1+S)$), we obtain the following algebraic equation for $S$:
\[
{t}^{2}-2\, \left( 4\,{S}^{2}+2\,S+1 \right) \left( 1+S \right) ^{2
}t-S \left( S+2 \right)  \left( 1+S \right) ^{4}=0,
\]
and $\Ffc(t)$ is rational in terms of $S$ and $t$. The series $S(t)$ is exactly the algebraic series, denoted $s(w)$, in Theorem 1 of~\cite{gao2001}, where a rational expression in terms of $t$ and $S(t)$ is given for the generating function of rooted 5-connected triangulations counted by the number of vertices minus $2$. We note that a combinatorial interpretation can be given for $-S$. Indeed, 
the expression of $B$ in~\eqref{eq:AB}
 is of the form $B=\bar{S}+B\,\bar{S}$, with $\bar{S}=t(1+2A+B^2)$. On the other hand, $B=-S/(1+S)$ ensures that $-S=\bar{S}$. 
 The radius of convergence for $-S$ is the same as for $B$, and as shown in~\cite{gao2001} it is also the radius of convergence of the generating function of rooted 5-connected triangulations 
 (the fact that the later radius of convergence is the same as the one of $\Ffc(t)$ can also be established by simple upper/lower bound relations between the counting coefficients, without an explicit computation of $\rho$).  
\hfill$\triangle$
\end{remark}

\subsection{Generating function of 5-connected triangulations with root-vertex degree $5$}

For $M$ a rooted 5c-triangulation, we denote by $v_1,\ldots,v_5$ the outer vertices in clockwise order around the outer face, with $v_1$ incident to the root corner. The \emph{augmented map} $\hat{M}$ of $M$ is the triangulation obtained by adding a vertex $v_{\infty}$ in the outer face, connected to vertices $v_1,\ldots,v_5$, and taking as new root corner the one incident to $v_{\infty}$ in the face $v_1,v_2,v_{\infty}$. 

\begin{lemma}\label{lem:5co}
For $M$ a rooted 5c-triangulation, the augmented map $\hat{M}$ is a 5-connected triangulation iff $M$ has all outer vertices of degree at least $4$ (i.e., there is no outer vertex of degree $3$). For $n\geq 6$ the mapping $M\to\hat{M}$ is a bijection from such rooted 5c-triangulations with $n$ vertices to rooted 5-connected triangulations with $n+1$ vertices and having root-vertex degree~$5$. 
\end{lemma} 
\begin{proof}
If $M$ has an outer vertex $v_i$ of degree $3$, then this vertex has degree $4$ in $\hat{M}$, hence $\hat{M}$ can not be $5$-connected.
If all outer vertices of $M$ have degree at least $4$, assume $\hat{M}$ is not 
5-connected. Then it has a separating 3-cycle or 4-cycle passing by $v_{\infty}$. 
There is actually no separating 3-cycle passing by $v_{\infty}$ as this would imply a chord in $M$, which would create a separating 3-cycle or 4-cycle in $M$. 
The fact that all outer vertices have degree at least $4$ easily implies that there is no inner vertex $v$ adjacent to two outer vertices $v_i,v_{i+2}$ in $M$; indeed the 4-cycle $v,v_i,v_{i+1},v_{i+2}$ would contain no vertex in its interior, since $M$ is a 5c-triangulation, and this would force $v_{i+1}$ to have $v$ as single inner neighbour, so that $v_{i+1}$ would have degree~$3$, a contradiction. 
Since there is no such inner vertex $v$, there can be no separating 4-cycle passing by $v_{\infty}$, and thus $\hat{M}$ is 5-connected.   

Finally, deleting the root-vertex in a 5-connected triangulation of root-vertex degree $5$ clearly yields a 5c-triangulation where all outer vertices have degree at least $4$. The mapping $M\to\hat{M}$ being clearly injective, we conclude that it gives a bijection between rooted 5c-triangulations with all outer vertices of degree at least $4$, and rooted 5-connected triangulations with root-vertex degree $5$. 
\end{proof}

A 5c-triangulation is called \emph{non-trivial} if it has at least two inner vertices. 
For $X\subseteq[1..5]$ we let $\cF_X$ be the family of non-trivial rooted 5c-triangulations 
such that the outer vertices of degree $3$ are the $v_i$ with $i\in X$. And we let $F_X(t)$ be the generating function of $\cF_X$ counted by inner vertices. 
A subset $X'\subseteq[1..5]$ is said to be \emph{shelling-compatible} with $X$, written $X'\uparrow X$, if $X'\cap X=\emptyset$ and every element in $X'$ is adjacent (in $\mathbb{Z}/5\mathbb{Z}$) to at least one element in $X$.

\begin{lemma}\label{lem:FX}
For $X\subseteq[1..5]$ we have $F_X(t)=0$ if $X$ contains a pair of adjacent elements in $\mathbb{Z}/5\mathbb{Z}$. If not and if $X$ is not empty, then we have
\begin{equation}\label{eq:FX}
F_X(t)=t^{|X|}\sum_{X'\uparrow X}F_{X'}(t).
\end{equation}
\end{lemma}  
\begin{proof}
Let $M$ be a non-trivial rooted 5c-triangulation. Assume there are consecutive outer vertices $v_i,v_{i+1}$ of degree $3$. Then there is an inner vertex $v$ that is neighbour of the 4 outer vertices $v_{i-1},v_i,v_{i+1},v_{i+2}$. Since $M$ is non-trivial, the 4-cycle $v_{i-1},v,v_{i+2},v_{i+3}$ contains at least one vertex in its interior, a contradiction. 

Let $X\subseteq[1..5]$ not containing adjacent elements and not empty. Let $M\in\cF_X$. 
Let $M'=\xi(M)$ be obtained from $M$ by deleting the outer vertices of degree $3$ and their incident edges (if the root-vertex $v_1$ is deleted then we root $M'$ at the new outer vertex that previously was the inner neighbour of $v_1$). Clearly $M'$ is a 5c-triangulation. And its set $X'$ of outer vertices of degree $3$ has to be disjoint from $X$; otherwise it would mean that $M$ has an outer vertex of degree $3$ whose unique inner neighbour has degree $4$, a contradiction with $M$ having no separating quadrangle. Also note that every $i\in X'$ has to be adjacent to some $j\in X$, otherwise $v_i$ would be an outer vertex of degree $3$ in $M$, so $i$ would be in $X$. 
Thus the mapping $\xi$ sends  $\cF_X$ to $\cup_{X'\uparrow X}\cF_{X'}$ and is clearly injective. 

It remains to check that $\xi$ is surjective. For $X'\uparrow X$ (with $X'$ not containing adjacent elements), let $M'\in\cF_{X'}$. Let $v_1',\ldots,v_5'$ be the outer vertices of $M'$. For each $i\in X$, add a new outer vertex $v_i$ of degree $3$ covering $v_i'$, i.e., connected to $v_{i-1}',v_i',v_{i+1}'$. Let $M$ be the resulting triangulation of the 5-gon (as before, if a new vertex $v_1$ has been created, we root $M$ at $v_1$). Assume $M$ has a separating 3-cycle or 4-cycle $\gamma$. Since $M'$ is a 5c-triangulation, this cycle $\gamma$ has to pass by a new outer vertex $v_i$ for $i\in X$. Clearly $\gamma$ can not be a 3-cycle otherwise there would be an inner edge connecting $v_{i-1}'$ and $v_{i+1}'$. If $\gamma$ is a 4-cycle, then it has to be of the form $v_{i-1}',v_i,v_{i+1}',v$ for some inner vertex $v$ of $M'$. Note that $v_{i-1}',v_i',v_{i+1}',v$ forms a 4-cycle of $M'$. Its interior contains no vertex since $M'$ is a 5c-triangulation. Hence $v_i'$ has to be of degree $3$ in $M'$ (with $v$ as its unique neighbour), a contradiction. Hence $M$ is a rooted 5c-triangulation. Clearly it is in $\cF_X$ and we have $M'=\xi(M)$.   
This concludes the proof that $\xi$ is a bijection from $\cF_X$ to $\cup_{X'\uparrow X}\cF_{X'}$, which yields~\eqref{eq:FX}. 
\end{proof}

\begin{prop}\label{prop:count_5co}
Let $c_n$ be the number of rooted 5-connected triangulations with $n+2$ vertices and root-vertex degree $5$, and let $\Ffco(t)=\sum_nc_nt^n$ be the associated generating function. Then 
\[
\Ffco(t)=\frac{t^4(1-3t+t^2)}{(1+t)^2}\big(\Ffc(t)-t\big),
\]
where $\Ffc(t)$ is expressed in Proposition~\ref{prop:count5c}.

And there is a positive algebraic constant $\kappa'\approx 0.00042228$
 such that $c_n= \frac{\kappa'}{2\sqrt{\pi}}\,\rho^{-n}n^{-5/2}$, for the same $\rho$ as in Proposition~\ref{prop:count5c}. 
\end{prop}
\begin{proof}
For $i\in\{0,1,2\}$ let $F^{(i)}(t)$ be the common generating function $F_X(t)$ for $|X|=i$
having no adjacent elements in $\mathbb{Z}/5\mathbb{Z}$. Then it follows from Lemma~\ref{lem:FX}
that
\[
F^{(2)}(t)=t^2\cdot\big(2F^{(2)}(t)+3F^{(1)}(t)+F^{(0)}(t)\big),\ \ \ \ F^{(1)}(t)=t\cdot\big(F^{(2)}(t)+2F^{(1)}(t)+F^{(0)}(t)\big) 
\]
This yields
\[
F^{(1)}(t)=\frac{t(1-t)}{1-3t+t^2}F^{(0)}(t),\ \ \ F^{(2)}(t)=\frac{t^2}{1-3t+t^2}F^{(0)}(t).
\]
Then we have
\[
\Ffc(t)-t=F^{(0)}(t)+5F^{(1)}(t)+5F^{(2)}(t)=\frac{(1+t)^2}{1-3t+t^2}F^{(0)}(t).
\]
Finally Lemma~\ref{lem:5co} ensures that $\Ffco(t)=t^4F^{(0)}(t)$. 

Regarding asymptotic enumeration, since the rational prefactor in $\Ffco(t)$ has no positive singularity, the series $\Ffco(t)$ has same singularity and singular type as $\Ffc(t)$, and we have $\kappa'=\frac{\rho^4(1-3\rho+\rho^2)}{(1+\rho)^2}\kappa\approx 0.00042228$.  
\end{proof}

\begin{remark}\label{rk:heaps}
Another way to see the relation between $\Ffc(t)$ and $\Ffco(t)$ is to observe that $\Ffc(t)-t=F^{(0)}(t)\cdot H(t)$ where $H(t)$ is the generating function of heaps of pieces on the 5-cycle graph, forbidding two same pieces being directly on top of each other. Thus $H(t)=\widetilde{H}(t/(1+t))$ where $\widetilde{H}(t)$ is the generating function of heaps of pieces on the 5-cycle graph, which by Viennot's involution~\cite{viennot2006heaps} is equal to $1/(1-5t+5t^2)$.  \hfill$\triangle$
\end{remark}

The series expansion starts as $\Ffco(t)=t^{10}+ 5\,t^{12}+ 10\,t^{13}+ 40\,t^{14}+ 131\,t^{15}+ 465\,t^{16}+ 1630\,t^{17}+ 5815\,t^{18}+ 20815\,t^{19}+74992\,t^{20}+\cdots$. The first term is given by the icosahedron, which has only one rooting. The second term is given by the 5-connected triangulation shown in Figure~9(a) of~\cite{gao2001}; it has $12$ vertices of degree $5$ and two vertices of degree $6$, the number of distinct rootings at degree $5$ (resp. degree $6$) vertices being $5$ (resp. $1$). The third term is given by the 5-connected triangulation shown in Figure~9(b) of~\cite{gao2001}; it has $12$ vertices of degree $5$ and three vertices of degree $6$, the number of distinct rootings at degree $5$ (resp. degree $6$) vertices being $10$ (resp. $3$).

\begin{remark}\label{rk:number_degree5}
Note that $6n\,c_n/5$ gives the number of rooted 5-connected triangulations with $n+2$ vertices and having a marked vertex of degree $5$. On the other hand, it is shown in~\cite{gao2001} that the number $d_n$ of rooted 5-connected triangulations with $n+2$ vertices has asymptotic equivalent $d_n\sim \frac{\kappa''}{2\sqrt{\pi}}\rho^{-n}n^{-5/2}$, where $\kappa''\approx 0.0010131$. Letting $X_n$ be the number of vertices of degree $5$ in a random rooted  5-connected triangulation with $n+2$ vertices we thus have $\mathbb{E}(X_n)=\frac{6nc_n/5}{d_n}\sim \xi n$, where $\xi=6\kappa'/(5\kappa'')\approx 0.50016$. 
 \hfill$\triangle$
\end{remark}

\section{Algorithmic applications and concluding remarks}\label{sec:algo}
\subsection{Random generation of 5-connected triangulations} 
Efficient bijective random generators are known for 3-connected (i.e., simple) triangulations~\cite{PS06} and for 4-connected (i.e., irreducible) triangulations~\cite{Fu07b}. It is also possible to derive a random generator for 4-connected triangulations from the one for 3-connected triangulations by the core-extraction method of~\cite{Schaeffer99}. However it seems difficult to apply this method to extract a random generator for 5-connected triangulations from the one for 4-connected triangulations; indeed the scheme relating 4-connected and 5-connected triangulations necessitates more general irreducible maps (with faces of degree 3 and 4) and a two-variable substitution scheme~\cite{gao2001}. Instead, we can more directly derive efficient random generators from our main bijection (Theorem~\ref{theo:main}).    

We recall that the Boltzmann distribution on a combinatorial class $\cC=\cup_n\cC_n$ assigns probability $x^n/C(x)$ to objects in $\cC_n$ for $n\geq 0$, where $x$ is any fixed positive value such that the generating function $C(x)$ converges. 
Via the sampling rules of~\cite{DuchonFLS04},
 the grammar~\eqref{eq:AB} shown in Figure~\ref{fig:decomp_AB} translates into 
Boltzmann samplers $\Gamma A(x)$ and $\Gamma B(x)$ for the tree families $\cA$ and $\cB$, and a Boltzmann sampler $\Gamma P(x)$ for the product class $\cP=\cA\times\cB$, such that the complexity of generating a tree is linear in its size. 
Then the singular generator $\Gamma P(\rho)$ combined with rejection~\cite[Sec.7.2]{DuchonFLS04} yields an exact-size uniform sampler for $\cP_n$ with complexity $O(n^2)$ (strategies to improve the running time are presented in~\cite{Spo21}) and an approximate-size sampler\footnote{In approximate-size sampling the size of the output is required to belong to the target-set $[n(1-\epsilon),n(1+\epsilon)]$ for some fixed $\epsilon\in(0,1)$.} for $\cP$  of complexity $O(n/\epsilon)$. 
Via the bijection $\Psi$ this gives an exact-size sampler $\Gamma F_{5c}[n]$ and an  approximate-size sampler $\Gamma F_{5c}[n,\epsilon]$ for rooted 5c-triangulations, 
with respective complexities $O(n^2)$ and $O(n/\epsilon)$, where the size is the number of inner vertices.

Regarding random generation in the family $\cFco=\cup_n(\cFco)_n$ of rooted 5-connected triangulations with root-vertex degree $5$ (with $n$ the number of vertices minus $2$),  
to obtain an exact-size sampler $\Gamma F_{5co}[n]$ we may repeatedly call $\Gamma F_{5c}[n-4]$ until the success situation where the generated object has no outer vertex of degree $3$. Then we obtain a (uniformly random) triangulation in $(\cFco)_n$ 
by the augmentation step of Lemma~\ref{lem:5co} adding a vertex $v_{\infty}$ connected to the outer vertices. The probability of success of each call to $\Gamma F_{5c}[n-4]$ is 
$\frac{[t^{n}]\Ffco(t)}{[t^{n-4}]\Ffc(t)}\sim \frac{1-3\rho+\rho^2}{(1+\rho)^2}\approx 0.20433$, the inverse of which gives the expected number of calls to $\Gamma F_{5c}[n-4]$ before success. Given a single call $M\leftarrow\Gamma F_{5c}[n-4]$, one may also  delete outer vertices of degree $3$ one by one until none remains on the outer contour, and then perform the augmentation step of Lemma~\ref{lem:5co}; the obtained object $M'=\chi(M)$ has (random) size $m=n-O(1)$ and is uniformly distributed on $(\cFco)_m$. 
 The fact that $n-m$ is $O(1)$ in probability follows from the relation $\Ffc(t)-t=H(t)\Ffco(t)$ of Remark~\ref{rk:heaps}, where $H(t)$ accounts for the deleted vertices, and from the fact that $H(t)$ is analytic at $\rho$. 
 This strategy is well-adapted to approximate-size sampling: rather than repeating calls to $\Gamma \Ffc[n,\epsilon])$ until the generated 5c-triangulation has no outer vertex of degree $3$, one can use a single call $M\leftarrow\Gamma \Ffc[n,\epsilon]$ and return $M'=\chi(M)$. With high probability the size $m$ of $M'$ remains in $[n(1-\epsilon),n(1+\epsilon]$ (if not, it is still very close from the lower bound). 

 Finally let us discuss how a random sampler $\Gamma F[n]$ can be obtained for rooted 5-connected triangulations with $n+2$ vertices, with no constraint on the degree of the  root-vertex. Given such a rooted triangulation, the operation of deleting the two edges of the root-face that are incident to the root-vertex clearly yields a rooted 5c-triangulation with the same number of vertices, and the mapping is injective. A rooted 5c-triangulation that is the image of a rooted 5-connected triangulation is called \emph{admissible}. 
 Testing if a rooted 5c-triangulation is admissible can easily be done in linear time (in the corresponding rooted triangulation, a separating 4-cycle has to pass by one of the two added  edges).    
 Then $\Gamma F[n]$ repeatedly calls $\Gamma \Ffc[n-3]$ until the generated object is admissible, then adds the edges $\{v_1,v_4\}$ and $\{v_1,v_3\}$. With the notation of Proposition~\ref{prop:count5c} and Remark~\ref{rk:number_degree5}, the probability of success of each call is $\frac{d_n}{[t^{n-3}]\Ffc(t)}
 \sim \frac{\kappa''}{\rho^3\kappa}\approx 0.12146$, whose inverse gives the expected number of calls to $\Gamma \Ffc[n-3]$ before success. Similarly, in approximate-size sampling, the generator  $\Gamma F[n,\epsilon]$ repeatedly calls $\Gamma \Ffc[n,\epsilon]$ until the generated object is admissible, and then adds  the edges $\{v_1,v_4\}$ and $\{v_1,v_3\}$.

\subsection{Optimal encoding of 5-connected triangulations} 
An (asymptotically) optimal encoder for a family $\cC=\cup_n\cC_n$ having exponential growth rate $\gamma>1$ is an injective mapping 
from $\cC$ to $\{0,1\}^*$ such that the maximal length $k_n$ over all encoding words of objects in $\cC_n$ satisfies $k_n\leq n\log(\gamma)+o(n)$ (in which case one must have $k_n\sim n\log(\gamma)$). Bijections for 3-connected triangulations yield an optimal encoding procedure~\cite{PS06,FuPoScL} (by reducing the problem to encoding quaternary trees), with application to mesh  compression. We have here $k_n\sim \alpha_3n$ with $\alpha_3=\log(256/27)\approx 3.30518$. 
Similary, the bijection for 4-connected triangulations~\cite{Fu07b} yields an optimal encoding procedure for this class (by reducing the problem to encoding ternary trees), giving $k_n\sim \alpha_4n$ with $\alpha_4=\log(27/4)\approx 2.75489$. 

Similarly we show here that our bijection yields an optimal encoder for 5-connected triangulations. The encoding of our trees is slightly more involved, as they have two types of nodes.  
Nevertheless, we can achieve optimality by following the ideas of~\cite{chottin1975demonstration,chottin1981enumeration,bacher2013multivariate} (which provide combinatorial proofs of multivariate Lagrange inversion formulas). Let $M$ be a 5-connected planar triangulation with $n$ vertices. We may mark and delete a vertex of degree $5$ to obtain a 5c-triangulation having $n-6$ inner vertices. Consider the associated leg-balanced 5-regular plane tree $T$. We may then mark a white node incident to a single inner edge. Cutting at the middle of this edge, we obtain a tree in $\cA$ with $n-7$ nodes. We have thus reduced the problem to finding an optimal encoder for the class $\cA=\cup_n\cA_n$ (the size $n$ being the number of nodes). 

Let $T\in\cA_n$. Note that the nodes of $T$ are of two types, A-nodes and B-nodes depending on the rightmost child being a node or a leaf. In each type there are 8 subtypes as shown at the top of  Figure~\ref{fig:coding_A_compact}. Let $n_j$ (resp. $n_j'$) be the number of A-nodes (resp. B-nodes) of $T$ having subtype $j\in[1..8]$. Letting $n_A=\sum_jn_j$ and $n_B=\sum_jn_j'$, note that we have the two compatibility conditions 
\begin{align*}
n_A&=1+n_1+2n_2+n_3+n_4+n_5+n_6+n_7+n_2'+n_3'+n_5'+n_6',\\
n_B&=n_3+n_4+2n_5+2n_6+2n_7+4n_8+n_4'+n_5'+n_6'+2n_7'+3n_8'.
\end{align*}
The 1st (resp. 2nd) line results from counting A-nodes (resp. B-nodes) according to their parent.   
We can encode $T$ by a pair $(w,w')$ of words where $w\in\frak{S}(a_1^{n_1}\cdots a_8^{n_8})$
and $w'\in\frak{S}(b_1^{n_1'}\cdots b_8^{n_8'})$, as follows. We order the A-nodes as $v_1,\ldots,v_{n_A}$ according to a left-to-right DFS traversal of $T$. Deleting the parent-edges of all non-root A-nodes yields an ordered forest $T_1,\ldots,T_{n_A}$ of trees such that  for each $i\in[1..n_A]$ the tree $T_i$ is rooted at $v_i$ and all other nodes of $T_i$ are B-nodes, as shown on the right side of Figure~\ref{fig:coding_A_compact}. We may then order the B-nodes of $T_i$ according to a left-to-right DFS traversal of $T_i$. Concatenating the ordered lists of B-nodes for $T_1,\ldots,T_{n_A}$ yields an ordered list $v_1',\ldots,v_{n_B}'$ of the B-nodes of $T$. The word $w=w_1\ldots w_{n_A}$ is then such that $w_i=a_j$ if $v_i$ has subtype $j\in[1..8]$. Similarly, the word $w'=w_1\ldots w_{n_B}$ is such $w_i'=b_j$ if $v_i'$ has subtype $j\in[1..8]$, see Figure~\ref{fig:coding_A_compact}. 

\fig{width=14cm}{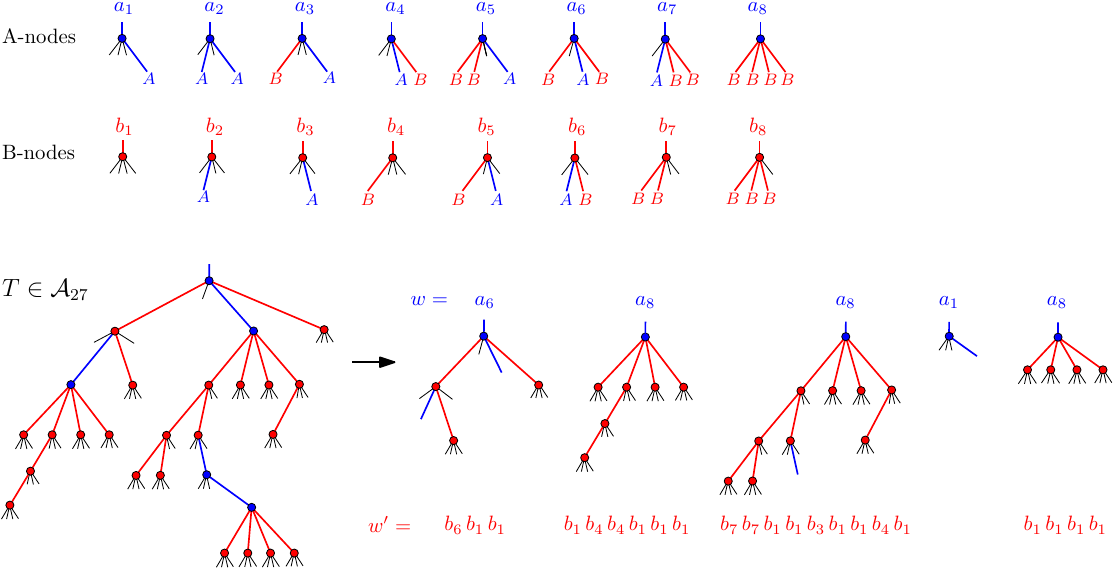}{Encoding of a tree $T\in\cA_{n}$ ($n=27$) by two words $w,w'$,  each on an 8-letter alphabet, with $|w|+|w'|=n$.}

Conversely, consider a pair of words $w,w'$ respectively on the alphabet $\{a_1,\ldots,a_8\}$ (resp. $\{b_1,\ldots,b_8\}$) whose letter multiplicities $n_j,n_j'$ add up to $n$ and satisfy the above two compatibility conditions. We first turn $w$ into an ordered list $v_1,\ldots,v_{n_A}$ of A-nodes, and turn $w'$ into an ordered list $v_1',\ldots,v_{n_B}'$ of B-nodes according to the correspondence in the upper part of Figure~\ref{fig:coding_A_compact}. Note that the A-spots and B-spots for children of these nodes are still unoccupied. 
Then, from the list $v_1',\ldots,v_{n_B}'$ we may iteratively construct an ordered forest of trees that contain only B-nodes whose B-spots are all occupied, and such that $v_1',\ldots,v_{n_B}'$ gives the ordered list of nodes in left-to-right DFS traversal of the trees. 
This classically requires a 
\L{}ukasiewicz condition: letting $s=n_4'+n_5'+n_6'+2n_7'+3n_8'$ be the total number of B-spots over $v_1',\ldots,v_{n_B}'$,  the number of B-spots in the strict prefix $v_1',\ldots,v_i'$ has to be strictly larger than $i+s-n_B$, for $1\leq i<n_B$. (This condition is met by a proportion $\frac{n_B-s}{n_B}=\frac{n_3+n_4+2n_5+2n_6+2n_7+4n_8}{n_B}$ of the words $w'\in\frak{S}(b_1^{n_1'}\cdots b_8^{n_8'})$.) If this condition is satisfied, we may then attach the $n_B-s$ obtained trees at the B-spots (ordered from left to right) of the A-nodes, to obtain an ordered forest $T_1,\ldots,T_{n_A}$ of trees such that $T_i$ is rooted at $v_i$, all other nodes of $T_i$ are B-nodes, and all B-spots are occupied. Finally we aggregate the trees into a single tree in $\cA_n$. Again this requires a  
\L{}ukasiewicz condition:  the total number of A-spots in $T_1,\ldots,T_i$ has to be at least $i$ for $1\leq i<n_A$. (This condition is met by a proportion $1/n_A$ of the pairs $w,w'$ with letter multiplicities $n_j,n_j'$ and such that $w'$ satisfies the first \L{}ukasiewicz condition.) 

The mapping from trees to word pairs and conversely from word pairs to trees clearly have linear time complexity. Note that a word pair $w,w'$ can be stored as an integer $n_A$ (giving the number of A-nodes) followed by a word of length $n$ on the $8$-letter alphabet $\{\ell_1,\ldots,\ell_8\}$. Thus, up to the negligible overhead integer $n_A$ of bit size $O(\log(n))$, we have an encoder with 3 bits per vertex. To achieve better compression rate, consider the letter multiplicities $n_j,n_j'$ of $w,w'$. Then, using e.g. Huffman codes with letters grouped into blocks of size $o(\log(n))$, one can encode $w,w'$ by a binary  word $W$ of length $\log\binom{n_A}{n_1,\ldots,n_8}+\log\binom{n_B}{n_1',\ldots,n_8'}+o(n)$, with linear time complexity for coding/decoding. By the above discussion, for compatible multiplicities $n_j,n_j'$ adding up to $n$, 
the number of word pairs $w,w'$ having these multiplicities and encoding trees in $\cA_n$ is   $\frac{n_3+n_4+2n_5+2n_6+2n_7+4n_8}{n_B}\frac1{n_A}\binom{n_A}{n_1,\ldots,n_8}\binom{n_B}{n_1',\ldots,n_8'}$. (The number $|\cA_n|$ is thus equal to the sum of all these contributions over compatible $n_j,n_j'$ adding up to $n$.)  
Hence we have 
$\log\binom{n_A}{n_1,\ldots,n_8}+\log\binom{n_B}{n_1',\ldots,n_8'}\leq \log(|\cA_n|)+O(\log(n))$. This ensures that the length of $W$ is at most $\log(|\cA_n|)+o(n)$. To summarize we have:

\begin{prop}
The above process gives an optimal encoder for 5-connected planar triangulations: the length of any coding word in size $n$ is bounded by $\alpha_5\,n+o(n)$ where $\alpha_5:=\log(1/\rho)\approx 2.013$.  
Coding and decoding have linear time complexity. 
\end{prop}

\subsection{Substitution approach and open questions on bijective extensions}

An explicit algebraic expression for the generating function of rooted 5-connected planar triangulations was obtained in~\cite{gao2001}, by a bivariate substitution method in the  extended setting of maps having both triangular and  quadrangular inner faces. 
To briefly summarize the approach, one can define for each $p\geq 3$ a \emph{$p$-map} as a simple planar map whose outer face contour is a simple cycle\footnote{In~\cite{gao2001} the root-face contour is not required to be simple; substitution relations are very similar in both models.} 
of length $p$, and whose inner faces have degree in $\{3,4\}$. In such a map, a \emph{separating cycle} is a cycle whose interior and exterior both contain at least one vertex. A \emph{$p$-map} is called \emph{irreducible} if it has no separating 3-cycle, and is called \emph{strongly irreducible} if it has no separating 3-cycle nor separating 4-cycle. 
One can then consider the generating function $M_p(x,y)$ of rooted $p$-maps with variable $x$ (resp. $y$) for the number of triangular (resp. quadrangular) inner faces; and similarly the generating function $\tM_p(x,y)$ (resp. $\sM_p(x,y)$) of irreducible (resp. strongly irreducible) rooted $p$-maps. The operation of emptying the maximal separating 3-cycles in a $p$-map yields the generating function relation
\begin{equation}\label{eq:irred}
M_p(x,y)=\tM_p(\tx,\ty)\ \mathrm{for}\ p\geq 4,\ \mathrm{where}\ \tx=M_3(x,y),\ \ty=y,
\end{equation}
and for $p=3$ the slightly more involved relation $M_3(x,y)-x=\tM_3(\tx,\ty)-\tx$. 
Then, the operation of emptying the maximal separating 4-cycles in an irreducible $p$-map yields the generating function relation
\begin{equation}\label{eq:strong_irred}
\tM_p(\tx,\ty)=\sM_p(\sx,\sy)\ \mathrm{for}\ p\geq 5,\ \mathrm{where}\ \sx=\tx,\ \sy=\tM_4(\tx,\ty)-2\tx^2.
\end{equation}
The substitution relations are more involved for $p=4$ (with a series-parallel decomposition on top of the substitution relation, as in~\cite{mullin1968enumeration}), and above all for $p=3$ which requires a careful inclusion-exclusion argument as explained in~\cite{gao2001}. The main reason for $p\geq 5$ being easier is that, in that case, a 4-cycle in an irreducible $p$-map is separating if and only if its interior contains a vertex, so that the maximal separating 4-cycles are interior-disjoint.  

As shown in~\cite{gao2001}, by solving a functional equation given by root-edge deletion, a Lagrangean expression\footnote{A bivariate generating function $f(x,y)$ is said to have a Lagrangean expression if it can be rationally expressed in terms of two series $u,v$ such that $x$ and $y$ are also rationally expressed in terms of $u,v$.} can be obtained respectively for $M_3(x,y)$ and for $M_4(x,y)$. Then a Lagrangean expression can also be computed for any $M_p(x,y)$ with $p\geq 5$, by induction on $p$, since root-edge deletion in a rooted $(p-2)$-map yields an explicit equation relating the counting series $M_i(x,y)$ with $i\in\{p-2,p-1,p,3,4\}$, having linear dependency on $M_p(x,y)$. For instance one gets
\begin{equation}\label{eq:M56}
M_5(x,y)=u(6u^4\!-5u^2v-\!10v^2\!+5v),\ \ M_6(x,y)=7u^6\!+\!17u^4v-4u^4\!-48u^2v^2\!+\!15u^2v-4v^3\!+3v^2,
\end{equation}
where 
\[
x=\frac{3u^3-2uv+u}{(1+v)^3},\ \ \ y=\frac{v-u^2}{(1+v)^3}.
\]

Then, via~\eqref{eq:irred} and~\eqref{eq:strong_irred}, for $p\geq 5$ the Lagrangean expression for $M_p(x,y)$ is turned into a Lagrangean expression for $\sM_p(\sx,\sy)$. The expression of $\sM_p(\sx,\sy)$ in terms of $u,v$ is the same\footnote{For $p\in\{3,4\}$ the expression changes due to more involved substitution relations.} 
as the one for $M_p(x,y)$, as given in~\eqref{eq:M56} for $p\in\{5,6\}$, only the expressions for $\sx,\sy$ are modified, to
\[
\sx=u-2uv+u^3,\ \ \ \sy=v\!-\!u^2\!-\!v^2\!+\!3u^2v\!-\!u^4\!-\!8u^2v^2\!-\!2u^6\!+\!8u^4v.
\]
One can extract the bivariate expansions of $u=u(\sx,\sy)$ and $v=v(\sx,\sy)$, and then the bivariate expansion of $\sM_p(\sx,\sy)$. By the Euler relation, the series $F_p(t,x,y):=t^{1-p/2}\sM_p(t^{1/2}x,t y)$ is the generating function of strongly irreducible rooted $p$-maps with variables $t,x,y$ for the numbers of inner vertices, triangular inner faces, and quadrangular inner faces. For instance we obtain
{\footnotesize
\begin{align*}
F_5(t,x,y)=&\ \left(5\,{x}^{3}+5\,xy\right)+ \left( {x}^{5}+5\,{x}^{3}y+5\,x{y}^{2} \right) t\ +\ 
 \left( 5\,{x}^{5}y+15\,{x}^{3}{y}^{2}+10\,x{y}^{3} \right) {t}^{2}\\
 &+\left( 10\,{x}^{7}y+50\,{x}^{5}{y}^{2}+70\,{x}^{3}{y}^{3}+25\,x{y}^{4
} \right) {t}^{3}+ \left( 25\,{x}^{9}y+185\,{x}^{7}{y}^{2}+430\,{x}^{5
}{y}^{3}+355\,{x}^{3}{y}^{4}+70\,x{y}^{5} \right) {t}^{4}\\
&+ \left( 65\,
{x}^{11}y+665\,{x}^{9}{y}^{2}+2280\,{x}^{7}{y}^{3}+3240\,{x}^{5}{y}^{4
}+1770\,{x}^{3}{y}^{5}+210\,x{y}^{6} \right) {t}^{5}\\
&+ \left( {x}^{15}\!+\!
190\,{x}^{13}y\!+\!2430\,{x}^{11}{y}^{2}\!+\!11180\,{x}^{9}{y}^{3}\!+\!23295\,{x}^
{7}{y}^{4}\!+\!22422\,{x}^{5}{y}^{5}\!+\!8550\,{x}^{3}{y}^{6}\!+\!660\,x{y}^{7}
 \right) {t}^{6}+\cdots
\end{align*}
\begin{align*}
F_6(t,x,y)=&\ \left(14{x}^{4}\!+\!21{x}^{2}y\!+\!3{y}^{2}\right)+ \left( 7x^{6}\!+\!36{x}^{4}y\!+\!39{x}^{2}{y}^{2}\!+\!2\,{y}^{3} \right) t+ \left( 3\,{x}^{8}\!+\!57\,{x}^{6}y\!+\!
159\,{x}^{4}{y}^{2}\!+\!114\,{x}^{2}{y}^{3}\!+\!3\,{y}^{4} \right) {t}^{2}
\\
 &+\left( 2\,{x}^{10}+114\,{x}^{8}y+558\,{x}^{6}{y}^{2}+844\,{x}^{4}{y}^
{3}+375\,{x}^{2}{y}^{4}+6\,{y}^{5} \right) {t}^{3}\\
&+ \left( 3\,{x}^{12}
+285\,{x}^{10}y+2055\,{x}^{8}{y}^{2}+5006\,{x}^{6}{y}^{3}+4665\,{x}^{4
}{y}^{4}+1302\,{x}^{2}{y}^{5}+14\,{y}^{6} \right) {t}^{4}
\\
&+ \left( 6\,{
x}^{14}+768\,{x}^{12}y+7551\,{x}^{10}{y}^{2}+26514\,{x}^{8}{y}^{3}+
40320\,{x}^{6}{y}^{4}+25416\,{x}^{4}{y}^{5}+4662\,{x}^{2}{y}^{6}+36\,{
y}^{7} \right) {t}^{5}+\cdots
\end{align*}
}
In particular, the series $F_5(t,1,0)-5$ coincides  with the series $F_{5c}(t)$ which we expressed in Proposition~\ref{prop:count5c}. 

It would be interesting to extend our bijective construction to strongly irreducible 5-maps, so as to obtain a combinatorial derivation of $F_5(t,x,y)$. These maps are characterized by the existence of (weighted) outdegree-constrained biorientations~\cite[Sec.12]{OB-EF-SL:Grand-Schnyder}, and the aim would be to show that the corresponding tree structures ---via the master bijection of~\cite{Bernardi-Fusy:dangulations}--- are characterized by local conditions only. 
 In another direction, one could try to extend our bijection to triangulated strongly irreducible $p$-maps, for $p\geq 6$. Such an extension has been carried out in bijections for simple triangulations and quadrangulations~\cite{albenque2013generic,Bernardi-Fusy:dangulations,PS06}.  
 Asking for the same two kinds of extensions can also be considered in the irreducible case, and appears as more tractable. The base case is that of triangulated irreducible 4-maps, so-called \emph{4c-triangulations}, for which a bijection (with ternary trees) is given in~\cite{Fu07b}. For the first extension, a combinatorial derivation of the generating function of rooted irreducible 4-maps via the method of slice decompositions is given in~\cite{BouttierG14} (which also gives the specification of the associated tree structures), and outdegree-constrained biorientations characterizing irreducible 4-maps are known~\cite[Sec.3.4]{OB-EF-SL:Grand-Schnyder}; for the second extension, a nice factorized formula for the counting coefficients suggests how to extend the bijective construction.          

Let us finally comment on the link between the enumeration of rooted 4c-triangulations (resp. 5c-triangulations) and the enumeration of rooted 4-connected (resp. 5-connected) triangulations.  
 The operation of adding a vertex of degree $4$ connected to the 4 outer vertices yields a bijection from rooted 4c-triangulations to rooted 4-connected triangulations with root-vertex degree $4$ (with one more vertex). Moreover, by root-edge deletion~\cite[Sec.4.3.2]{Fu07b} one can rationally express the generating function of rooted 4-connected triangulations in terms of $t$ and the generating function of rooted 4c-triangulations, which has a bijective derivation. The situation in the 5-connected case is a bit more involved. First, as seen in Lemma~\ref{lem:5co}, one needs to restrict rooted 5c-triangulations to a subfamily ---those with no outer vertex of degree $3$--- so that the operation of adding a vertex of degree $5$ connected to the 5 outer vertices yields a bijection to rooted 5-connected triangulations with root-vertex degree $5$. Nevertheless, as seen in Proposition~\ref{prop:count_5co}, this  generating function of this subfamily is just the the generating function $\Ffc(t)=F_5(t,1,0)-5$ of the whole family multiplied by a simple rational prefactor. What seems more difficult is to obtain a rational expression for the generating function $F_3(t,1,0)$ of rooted 5-connected triangulations in terms of $t$ and $F_5(t,1,0)$.  
 It however should be possible (though a bit involved, details omitted) to obtain by root-edge deletion a rational expression for $F_3(t,1,0)$ in terms of $\{t,F_5(t,1,0),F_6(t,1,0)\}$, thus giving special motivation for finding a bijective derivation of $F_6(t,1,0)$.  

\bigskip

\noindent{\bf Acknowledgments.} The author is very grateful to Olivier Bernardi and Shizhe Liang for the numerous fruitful discussions and collaboration~\cite{BernardiFL23,OB-EF-SL:Grand-Schnyder} on which this work builds. Partially supported by the project ANR-23-CE48-0018 (CartesEtPlus), and the project ANR-20-CE48-0018 (3DMaps).

\bibliographystyle{alpha}
\bibliography{biblio-Schnyder}

\end{document}